\documentclass[12pt, a4paper, twoside]{amsart}

\usepackage[margin=2.85cm]{geometry}

\usepackage[T1]{fontenc}
\usepackage[utf8]{inputenc}
\usepackage{amsmath,amsthm,amsfonts,amssymb}
\usepackage{bbm,bm,ae,microtype}
\usepackage[UKenglish]{babel}
\usepackage{hyperref, url}

\theoremstyle{definition}

\theoremstyle{remark}

\theoremstyle{plain}
\newtheorem{thm}{Theorem}
\newtheorem{lem}[thm]{Lemma}

\newtheorem{cor}[thm]{Corollary}

\DeclareMathOperator{\diam}{diam}
\DeclareMathOperator{\Bad}{Bad}
\DeclareMathOperator{\supp}{supp}
\DeclareMathOperator{\dimh}{dim_H}
\DeclareMathOperator{\hdim}{dim_H}

\author{S. Kristensen}
\address{S. Kristensen, Department of Mathematics, Aarhus University, Ny Munkegade 118, DK-8000 Aarhus C, Denmark}
\email{sik@math.au.dk}

\author{T. Persson}
\address{T. Persson, Centre for Mathematical Sciences, Lund University, Box 118, SE-22100 Lund, Sweden}
\email{tomasp@maths.lth.se}

\thanks{We thank Niclas Technau for pointing out to us that the
  paper \cite{chow2023dispersion} contains a result that implies a
  previous version of our Theorem~\ref{thm:measure}. This lead
  to the present and stronger version of
  Theorem~\ref{thm:measure}. We also thank an anonymous referee for a careful reading of the original manuscript. SK's research is supported by the
  Independent Research Fund Denmark (Grant ref. 1026-00081B) and
  the Aarhus University Research Foundation (Grant
  ref. AUFF-E-2021-9-20)}
	

\begin{document}

\title{On the distribution of sequences of the form $(q_ny)$}

\begin{abstract}
  We study the distribution of sequences of the form
  $(q_ny)_{n=1}^\infty$, where $(q_n)_{n=1}^\infty$ is some
  increasing sequence of integers. In particular, we study the
  Lebesgue measure and find bounds on the Hausdorff dimension of
  the set of points $\gamma \in [0,1)$ which are well
  approximated by points in the sequence
  $(q_ny)_{n=1}^\infty$. The bounds on Hausdorff dimension are
  valid for almost every $y$ in the support of a measure of
  positive Fourier dimension. When the required rate of
  approximation is very good or if our sequence is sufficiently
  rapidly growing, our dimension bounds are sharp. If the measure
  of positive Fourier dimension is itself Lebesgue measure, our
  measure bounds are also sharp for a very large class of
  sequences. 
\end{abstract}

\maketitle

\section{Introduction}

Let $y$ be fixed real number and let $(q_n)_{n=1}^\infty$ be a sequence of positive integers. We consider the set
\begin{equation}
  \label{eq:set}
  W_{y,\alpha} = \{\, \gamma \in [0,1) : \Vert q_ny - \gamma\Vert <
  n^{-\alpha} \text{ for infinitely many } n \in \mathbb{N} \,\},
\end{equation}
where $\Vert \cdot \Vert$ denotes the distance to the nearest integer.

The study of this set is not a new one. For instance, if
$(q_n)_{n=1}^\infty$ grows exponentially but not
super-exponentially fast and $\alpha < \frac{1}{2}$, it is shown
in \cite{MR3133988} that for almost every $y$ with respect to any
measure with sufficiently nice Fourier transform, the set
$W_{y,\alpha}$ is in fact the entire interval $[0,1)$.

In this paper, we continue this study using results and ideas
from \cite{MR3133988} and \cite{MR3341452}, in particular a
certain discrepance estimate \cite[Corollary~7]{MR3133988} and a
lemma on Hausdorff dimension, Lemma~\ref{lem:perssonreeve}.  The
motivation in \cite{MR3133988} for studying this set was
applications to certain Littlewood-type problems. More
concretely, a version of the inhomogeneous Littlewood conjecture
states that for all $x \in \mathbb{R}$, $y > 0$ irrational and
$\gamma \in \mathbb{R}$,
\[
  \liminf_{q \rightarrow \infty} q \Vert qx \Vert \Vert qy -
  \gamma \Vert = 0,
\]
where the liminf is over all positive integers. A consequence of
the main result in \cite{MR3133988} is that for any fixed badly
approximable number $x \in \Bad$, there is a set $G \subseteq \Bad$ of
maximal Hausdorff dimension such that for all $y \in G$ and all
$\gamma \in \mathbb{R}$,
\begin{equation}
  \label{eq:HJK}
  q \Vert qx \Vert \Vert qy - \gamma \Vert < \frac{1}{(\log
    q)^{1/2 - \epsilon}} \qquad  \text{for infinitely many
    integers } q.
\end{equation}
As always, the set $\Bad$ is the set of badly approximable numbers, i.e.
\[
  \Bad = \Bigl\{ \, x \in [0,1) : \liminf_{q \rightarrow \infty}
  q \Vert qx \Vert = 0 \, \Bigr\}.
\]
Of course, this set is equal to the set of numbers for which the
sequence of partial quotients in the simple continued fraction
expansion is bounded, see e.g.\ \cite{MR161833}.

In fact, for $x \in \Bad$, \eqref{eq:HJK} was shown to hold for
$\mu$-almost every $y$ in the support of a probability measure
$\mu$ of positive Fourier dimension, a notion which we will
define along with the notion of Hausdorff dimension in the next
section.

In fact, the right hand side of \eqref{eq:HJK} was improved by
Chow and Zafeiropoulos \cite{MR4258964} to
$(\log\log\log q)^{\epsilon + 1/2}/(\log q)^{1/2}$, though one
should note that the power of the logarithm in the denominator is
$1/2$. This reflects in the exponent $\alpha$ in the definition
of the set $W_{y,\alpha}$. Indeed, the exponent $1/2$ is a
limiting case for which the methods of \cite{MR3133988} give
\eqref{eq:HJK} which comes from a discrepancy bound. This in turn
works only for lacunary sequences, and can only give an exponent
smaller than $1/2$. In this case, the discrepancy bound shows
that $W_{y,\alpha} = [0,1)$, so that \eqref{eq:HJK} holds. It
appears to be very difficult to prove properties of this set for
$\alpha \ge 1/2$ as remarked in \cite{MR4425845}. In fact, it
follows from the work of Technau and Zafeiropoulos
\cite{MR4112677} that methods involving the discrepancy of
sequences of the form $(q_n \alpha)_{n=1}^\infty$ will not
accomplish this.

We are able to obtain some results on the size of $W_{y,\alpha}$
for $\alpha$ in the range between $1/2$ and $1$, although we
suspect that these are not best possible. For $\alpha < 1$ and
some non-lacunary (and for lacunary) sequences $(q_n)$ we are
able to prove that for $\mu$-almost every $y$ in the support of a
probability measure of positive Fourier dimension, the set
$W_{y,\alpha}$ has positive Lebesgue measure, see
Theorem~\ref{thm:measure}. If the sequence $(q_n)$ is lacunary
and $\alpha = 1$, then Chow and Technau
\cite[Theorem~1.5]{chow2023dispersion} showed that
$W_{y,\alpha} = [0,1)$ for $\mu$-almost every $y$.

If $\mu$ is the Lebesgue measure, we are able to relax the growth
condition on the sequence $(q_n)$ to being just a little more
than linear, while at the same time prove that the Lebesgue
measure of $W_{y,\alpha}$ is equal to $1$, see
Theorem~\ref{thm:Lmeasure}. Of course, the result of Chow and
Technau implies this and more when $(q_n)$ is lacunary, but our
result allows for less restrictive growth conditions.

For $\alpha > 1$, an application of the Borel--Cantelli lemma
immediately shows that the set is Lebesgue null for all
$x$. However, it is not usually empty, and in this note we will
derive some properties of the set $W_{y,\alpha}$ in the
$\alpha > 1$ range. A special case of our results is that for
Lebesgue almost all $y$, the set has Hausdorff dimension
$\frac{1}{\alpha}$, though our result in Theorem~\ref{thm:Hdim}
is more general. We will state it precisely, once we have
provided the technical background for doing so.

In the above mentioned result on the Hausdorff dimension, it is
critical that the sequence $(q_n)_{n=1}^\infty$ grows
geometrically fast. Nonetheless, the problem of studying the
distribution of orbits of fractional parts of $q_ny$ for such
sequences has attracted quite a bit of attention. For instance,
if the sequence $(q_n)_{n=1}^\infty$ contains all numbers of the
form $2^i3^j$ in increasing order, Furstenberg \cite{MR213508}
famously showed that orbits are either finite or dense and
conjectured that dense orbits are uniformly
distributed. Similarly, initiated by the work of Rudnick and
Sarnak \cite{MR1628282}, the distribution of sequences of the
form $(n^d y)_{n=1}^\infty$ has attracted a considerable amount
of attention. While we do not claim to solve any of the famous
and important open problems about the distribution of these
sequences, we are able to find upper and lower bounds on the
Hausdorff dimension of $W_{y,\alpha}$ regardless of any growth
conditions (Theorem~\ref{thm:Hdim2}), which are valid for almost all
$y$ in an appropriate sense. Specialising our main result to the
Lebesgue measure, we find that for $\alpha > 1$ and for Lebesgue
almost every $y$, the Hausdorff dimension of $W_{y,\alpha}$ is
somewhere in between $\frac1{2\alpha}$ and $\frac{1}{\alpha}$.

The paper is structured as follows. In the next section, we will
provide the relevant definitions of notions from uniform
distribution theory and fractal geometry. With these notions in
place, we will state our main theorems in
Section~\ref{sec:statements}, where we will also derive some
corollaries of these. The main theorems are proved in
Section~\ref{sec:proofs}. We end the paper with some concluding
remarks.

\section{Some technical background}
\label{sec:background}

We will let $\lambda$ denote the Lebesgue measure on $[0,1]$.
Our approach depends heavily on the notion of the discrepancy of
a sequence, which is a measure of how far the sequence is from
being uniformly distributed. For a sequence $(x_n)_{n=1}^\infty$
in $[0,1)$ and $N \in \mathbb{N}$, we define the $N$-discrepancy
of the sequence by
\[
  D_N(x_n) = \sup_{I \subseteq [0,1)} \vert \# \{\, n \le N : x_n
  \in I \,\} - N\,\lambda(I)\vert,
\]
where the \emph{supremum} is taken over all sub-intervals of
$[0,1)$. The sequence $(x_n)_{n=1}^\infty$ is said to be
uniformly distributed in $[0,1)$ if $D_N(x_n) = o(N)$.

A word of caution is appropriate here. Some authors define the discrepancy of a sequence to be the above quantity divided by $N$. In this notion of discrepancy, the quantity can perhaps most easily be understood to be a speed of convergence of the probability measures $\frac{1}{N}\sum_{n=1}^N \delta_{x_n}$ to the Lebesgue measure. The former notion favours counting over measure theory. The two notions are of course entirely equivalent, but the double meaning of the word `discrepancy' can cause confusion. \emph{In this paper, discrepancy is taken to mean the former.}

Many of our results deal with Hausdorff dimension, which we now define. For a set $E \subseteq \mathbb{R}^n$ and real numbers $s \ge 0$ and $\delta > 0$, consider the quantity,
\[
\mathcal{H}_\delta^s (E) = \inf_{\mathcal{C}_\delta} \sum_{U \in \mathcal{C_\delta}} \diam(U)^s,
\]
where the \emph{infimum} is taken over all covers $\mathcal{C}_\delta$ of $E$ with sets of diameter at most $\delta$. As $\delta$ decreases, the collection of such covers becomes smaller, so that $\mathcal{H}_\delta^s (E)$ must increase. Allowing infinite limits, this means that
\[
\mathcal{H}^s (E) = \lim_{\delta \rightarrow 0} \mathcal{H}_\delta^s (E)
\]
exists. In fact, it is a regular Borel measure, and for all but
possibly one value of $s$, the value attained by
$\mathcal{H}^s (E)$ is either $0$ or $\infty$. Furthermore,
$\mathcal{H}^s (E) = 0$ for $s > n$. Since $\mathcal{H}^s (E)$
evidently decreases as $s$ increases, we can define the Hausdorff
dimension of $E$ to be the `breaking point', i.e.
\[
\dimh(E) = \inf \{\, s \ge 0 : \mathcal{H}^s (E) = 0 \,\}.
\]

A particularly pleasing type of sets are intersective sets in the
sense of Falconer, see \cite{MR1260112}. For $s > 0$, the
Falconer class $\mathcal{G}^s$ is defined as the collection of
$G_\delta$-sets $F$ of $\mathbb{R}^n$ for which
\[
\dimh \biggl( \bigcap_{i=1}^\infty f_i(F) \biggr) \ge s,
\]
for any countable sequence $(f_i)_{i=1}^\infty$ of similarities,
and we say that $E$ is intersective if $E \in \mathcal{G}^s$ for
some $s > 0$. Among the desirable properties of the families
$\mathcal{G}^s$ we find the fact that it is closed under
countable intersections.

We will consider the intersection classes $\mathcal{G}^s$
restricted to $[0,1)$, which we denote by
$\mathcal{G}^s ([0,1))$. A set $E$ belongs to
$\mathcal{G}^s ([0,1))$ if there is a set $F \in \mathcal{G}^s$
such that $E = [0,1) \cap F$. Sets in $\mathcal{G}^s ([0,1))$
have Hausdorff dimension at least $s$, and
$\mathcal{G}^s ([0,1))$ is closed under countable intersections.

A final notion, which will play an important role is the Fourier
dimension of a measure, see e.g.\ \cite{MR3617376}. The Fourier
dimension of a Radon measure $\mu$ on $\mathbb{R}$ is defined as
\begin{equation}
  \label{def:Fdim}
  \dim_F(\mu) = \sup \{\, s \in [0,1] : \vert \hat{\mu}(\xi) \vert
  = O(|\xi|^{-s/2}) \,\}.
\end{equation}
Here, $\hat{\mu}$ is the Fourier transform of the measure defined as usual to be
\[
\hat{\mu}(\xi) = \int e^{2\pi i x \xi} \, \mathrm{d}\mu(x).
\]
The terminology originates in the fact that for a set $E \subseteq \mathbb{R}$, any measure supported on $E$ has Fourier dimension at most equal to the Hausdorff dimension of $E$. The Fourier dimension of the set $E$ is the supremum of the Fourier dimensions of the Radon measures supported on $E$.

Note that a set $E$ need not have the same Hausdorff and Fourier dimension. For instance, the ternary Cantor set has Hausdorff dimension $\log 2/ \log 3$ by a classical calculation found in all textbooks on fractal geometry. However, it was shown by Kahane and Salem \cite{MR160065} that the Fourier dimension of the ternary Cantor set is equal to zero. This result in fact follows rather easily from the results of both \cite{MR3133988} and this paper.

\section{Statement of results}
\label{sec:statements}

From \cite{MR3133988}, we know that if $\mu$ is some measure
supported on $[0,1)$ of positive Fourier dimension, then for
$\mu$-almost every $y$, there is a positive number $\nu$, such
that $W_{y,\alpha} = [0,1)$ for $\alpha < \nu$. If we further
assume that $(q_n)_{n=1}^\infty$ is lacunary, then Chow and
Technau \cite[Theorem~1.11]{chow2023dispersion} recently showed
that this holds true for any $\alpha < 1$. In fact they obtained
an even stronger result: Changing the rate $n^{-\alpha}$ to
$n^{-1} (\log n)^{3 + \varepsilon}$, they proved that
\[
  \{\, \gamma \in [0,1) : \Vert q_ny - \gamma\Vert < n^{-1}(\log
  n)^{3+\varepsilon} \text{ for infinitely many } n \in
    \mathbb{N} \,\} = [0,1)
\]
for almost every $y$.

Our first results is valid also for some non-lacunary sequences
including sequences of the form $q_n = n^d$.

\begin{thm}
  \label{thm:measure}
  Suppose that $\mu$ is a probability measure on $\mathbb{R}$ of
  positive Fourier dimension with decay
  $|\hat{\mu} (\xi)| = O(|\xi|^{-\tau})$. If $(q_n)_{n=1}^\infty$
  is a sequence of integers with
  $|q_n - q_m| > c |n-m|^{\frac{1}{\tau} + \varepsilon}$ for some
  $c, \varepsilon > 0$, and $\alpha < 1$, then
  $\lambda (W_{y,\alpha}) > 0$ for $\mu$-almost every $y$.
\end{thm}

In particular, Theorem~\ref{thm:measure} applies to the case when
$\mu$ is the Lebesgue measure and $(q_n)$ is given by an integer
polynomial of degree at least $3$.

Specialising $\mu$ to be the Lebesgue measure on $[0,1]$, we can
obtain the following stronger result.

\begin{thm}
  \label{thm:Lmeasure}
  Suppose that $(q_n)_{n=1}^\infty$ is a sequence of integers
  with $|q_n - q_m| > c |n-m|^{1 + \varepsilon}$. For
  $\alpha < 1$, the set $W_{y,\alpha}$ is Lebesgue full for
  Lebesgue-almost every $y$.
\end{thm}

We should remark that the growth condition in Theorem \ref{thm:measure} can be relaxed to $|q_n - q_m| > c |n-m|^{\frac{1}{\tau}} (\log |n-m|)^{1+\varepsilon}$ and similarly in Theorem \ref{thm:Lmeasure}. As it will be apparent in the proofs, all that matters is that a certain series is convergent.

We are not able to push the results for $\alpha < 1$ any further,
and we are not able to say anything about the measure of
$W_{y,\alpha}$ when $\alpha = 1$. Note however that the result of
Chow and Technau \cite[Theorem~1.5]{chow2023dispersion} implies
that $W_{y,\alpha} = [0,1)$ for $\alpha = 1$ and for Lebesgue
almost all $y$, provided the sequence $q_n$ is lacunary.  However
as already mentioned, for $\alpha > 1$ the first Borel--Cantelli
lemma immediately implies that the set $W_{y,\alpha}$ is a
Lebesgue null set, and it is of interest to calculate its
Hausdorff dimension. We are able to get a precise result valid
almost surely in the case of lacunary sequences.

\begin{thm}
  \label{thm:Hdim}
  Let $\mu$ be a probability measure on $\mathbb{R}$ with positive
  Fourier dimension, and let $(q_n)_{n = 1}^\infty$ be a lacunary
  sequence of integers. For any $\alpha \ge 1$, for $\mu$-almost
  all $y$,
  \[
    \dimh (W_{y,\alpha}) = \frac{1}{\alpha}.
  \]
  In fact, the set belongs to Falconer's class
  $\mathcal{G}^\frac{1}{\alpha} ([0,1))$.
\end{thm}

For non-lacunary sequences $(q_n)_{n=1}^\infty$, our results are
less precise. Nonetheless, the problem of studying the
distribution of orbits of fractional parts of $q_ny$ for such
sequences has attracted quite a bit of attention. For instance,
if the sequence $(q_n)_{n=1}^\infty$ contains all numbers of the
form $2^i3^j$ in increasing order, Furstenberg \cite{MR213508}
famously showed that orbits are either finite or dense and
conjectured that dense orbits are uniformly
distributed. Similarly, initiated by the work of Rudnick and
Sarnak \cite{MR1628282}, the distribution of sequences of the
form $(n^d y)_{n=1}^\infty$ has attracted a considerable amount
of attention. While we do not claim to solve any of the famous
and important open problems about the distribution of these
sequences, we will prove the following.

\begin{thm}
  \label{thm:Hdim2}
  Let $(x_n)_{n=1}^\infty$ be a sequence in $[0,1)$ with
  $D_N((x_n)) = O(N^{1-\eta})$ for some $\eta \in (0,1)$. Then the
  set
  \[
    W = \{\, \gamma \in [0,1) : \Vert x_n - \gamma \Vert <
    n^{-\alpha} \text{ infinitely often}\,\}
  \]
  satisfies that
  $\frac{\eta}{\alpha} \le \dimh W \le \frac{1}{\alpha}$ for
  $\alpha \geq 1$.
\end{thm}

In \cite{MR3133988}, it is shown that for an increasing sequence
of integers $(q_n)_{n=1}^\infty$, the sequence
$(x_n)_{n=1}^\infty = (\{q_n y\})_{n=1}^\infty$ satisfies the
required discrepancy estimate for $\mu$-almost every $y$,
whenever $\mu$ is a measure of positive Fourier dimension. More
precisely, it follows from the proof of Corollary~7 in
\cite{MR3133988} that if $|\hat{\mu} (\xi)| = O(|\xi|^{-2\eta})$,
then $D_N((x_n)) = O(N^{1-\eta'})$ for any $\eta' < \eta$.  Thus,
we immediately have the following corollary.

\begin{cor}
  \label{cor:Hdim}
  Let $\mu$ be a probability measure on $[0,1]$ with Fourier
  decay $|\hat{\mu} (\xi)| = O(|\xi|^{-2\eta})$, and let
  $(q_n)_{n = 1}^\infty$ be an increasing sequence of
  integers. For any $\alpha \geq 1$, for $\mu$-almost all $y$,
  \[
    \frac{\eta}{\alpha} \le \dimh W_{y,\alpha} \le
    \frac{1}{\alpha}.
  \]
\end{cor}

Note in particular that although the Fourier dimension of the
Lebesgue measure $\lambda$ is 1, we have the stronger Fourier
decay $|\hat{\lambda} (\xi)| = O (|\xi|^{-1})$. (This is because
in the definition \eqref{def:Fdim} of the Fourier dimension, the
supremum is taken over the interval $[0,1]$.) Hence, in the case
of Lebesgue measure, Corollary~\ref{cor:Hdim} implies that
$\frac{1}{2\alpha} \leq \dimh W_{y,\alpha} \leq \frac{1}{\alpha}$
for Lebesgue almost every $y$.

Theorem~\ref{thm:Hdim2} does not give a precise Hausdorff
dimension of the set $W$, but under an additional assumption on
the sequence $(x_n)_{n=1}^\infty$, we do in fact get the exact
dimension. If we define the quantity,
\[
  d_N ((x_n)) = \inf_{1 \leq k < l \leq N} \lVert x_k - x_l \rVert,
\]
we are able to prove the following. 

\begin{thm}
  \label{thm:extraassumption}
  Let $(x_n)$ be a sequence in $[0,1)$ with
  $D_N ((x_n)) \leq c N^{1-\eta}$ and
  $d_N ((x_n)) \geq c N^{-\beta}$, where $0 < \eta < 1$ and
  $1 \leq \beta \leq \eta (\alpha - 1) + 1$. Put
  \[
    W = \{\, y \in [0,1) : \lVert x_n - y \rVert < n^{-\alpha}
    \text{ infinitely often} \,\}.
  \]
  Then $\dimh W = \frac{1}{\alpha}$.
\end{thm}

We get the following corollary to Theorem~\ref{thm:extraassumption}.

\begin{cor}
  \label{cor:lastcor}
  Let $(q_n)_{n=1}^\infty$ be an increasing sequence of
  integers. For any $\alpha \geq 3$ and $\lambda$-almost every
  $y$, we have
  \[
    \dimh W_{y,\alpha} = \frac{1}{\alpha}.
  \]
\end{cor}

A proof of Corollary~\ref{cor:lastcor} is in the end of
Section~\ref{sec:proofs}.

\section{Proofs}
\label{sec:proofs}

We first prove Theorem~\ref{thm:measure}.

\begin{proof}[Proof of Theorem~\ref{thm:measure}]
  Let $B (x,r) = \{\, y \in \mathbb{R} : \lVert x - y \rVert
  < r \,\}$. We put
  \[
    A_k = A_k (y) = \{\, \gamma \in [0,1) : \lVert q_k y - \gamma
    \rVert < k^{-\alpha} \, \} = B(q_k y, k^{-\alpha}) \cap [0,1).
  \]
  Then
  \[
    \lambda (A_k) = 2 k^{-\alpha}.
  \]
  We also have for $k < l$ that
  \begin{equation} \label{eq:intersection}
    \lambda (A_k \cap A_l) \leq 2 l^{-\alpha} \mathbbm{1}_{B(0,
      2k^{-\alpha})} ((q_k - q_l)y)
  \end{equation}
  since
  $\lambda (A_k \cap A_l) \leq \lambda(A_l) = 2 l^{-\alpha}$
  always holds and $A_k \cap A_l = \emptyset$ holds when
  $(q_k y - q_l y) \not \in B(0, 2k^{-\alpha})$.

  We now consider the average
  $\int \lambda (A_k(y) \cap A_l (y)) \, \mathrm{d} \mu
  (y)$. Writing $\mathbbm{1}_{B(0, 2k^{-\alpha})}$ as a Fourier
  series,
  \[
    \mathbbm{1}_{B (0, 2k^{-\alpha})} (y) = \sum_{j \in
      \mathbb{Z}} a_j e^{i 2 \pi j y},
  \]
  we can estimate the Fourier coefficients by
  $|a_j| \leq 1 / |j|$ for $j \neq 0$ and we have
  $a_0 = 4 k^{-\alpha}$.
  This implies that
  \[
    \int \mathbbm{1}_{B(0, 2k^{-\alpha})} ((q_k - q_l)y) \,
    \mathrm{d} \mu (y) = \sum_j a_j \int e^{i 2 \pi j (q_k - q_l)
      y} \, \mathrm{d} \mu \leq a_0 + \sum_{j \neq 1} j^{-1}
    \hat{\mu} (j (q_k - q_l)).
  \]
  Since $\mu$ has positive Fourier dimension with decay rate
  $\tau$, and by the assumption on the sequence $(q_n)$, we have
  $\hat{\mu} (j (q_k - q_l)) \leq \tilde{c}_1 j^{-\tau}
  (l-k)^{-1 - \tau \varepsilon}$. Therefore, we have
  \[
    \int \mathbbm{1}_{B(0, 2k^{-\alpha})} ((q_k - q_l)y) \,
    \mathrm{d} \mu (y) \leq 4 (k^{-\alpha} + c_1 (l-k)^{-1 - \tau
      \varepsilon})
  \]
  for some constant $c_1$. By \eqref{eq:intersection} we then
  have
  \begin{equation} \label{eq:average} \int \lambda(A_k \cap A_l)
    \, \mathrm{d} \mu \leq 8 l^{-\alpha} (k^{-\alpha} + c_1
    (l-k)^{-1 - \tau\varepsilon}).
  \end{equation}

  Take $m < n$ and put
  \[
    S_{m,n} (y) = S_{m,n} = \sum_{k=m}^n \lambda (A_k), \qquad
    C_{m,n} (y) = \sum_{m\leq k,l \leq n} \lambda (A_k(y) \cap
    A_l (y)),
  \]
  We consider the sets 
  \[
    \Delta_{m,n} (p) = \biggl\{\, y : C_{m,n} (y) > p \int C_{m,n} (y)
    \, \mathrm{d} \mu (y) \,\biggr\}
  \]
  For $p \geq 1$, we have $\mu (\Delta_{m,n} (p)) \leq
  p^{-1}$. Therefore, $\mu (\complement \Delta_{m,n} (p)) \geq 1
  - p^{-1}$ and
  \[
    \complement \Delta_{m,n} (p) = \biggl\{\, y : C_{m,n} (y) \leq p
    \int C_{m,n} (y) \, \mathrm{d} \mu (y) \,\biggr\}.
  \]

  Fix $p > 1$ and take a quickly increasing sequence $n_j$ going
  to infinity in such a way that
  $\frac{n_{j+1}}{n_j} \to \infty$. Put
  $\Delta_j (p) = \Delta_{n_j, n_{j+1}} (p)$. Then
  $G (p) = \limsup \complement \Delta_j (p)$ satisfies
  $\mu (G (p)) \geq 1 - p^{-1}$.

  Take $y \in G (p)$. Then there are infinitely many $j$ for
  which
  \[
    C_{n_j,n_{j+1}} (y) \leq p C_j,
  \]
  where $C_j = \int C_{n_j,n_{j+1}} \, \mathrm{d}\mu$.

  Put $S_j = S_{n_j,n_{j+1}}$ We estimate $C_j$ and $S_j$. By
  \eqref{eq:average} we have
    \begin{align}
      C_j
      & \leq S_j + 2 \sum_{l = n_j + 1}^{n_{j+1}} \sum_{k = n_j}^{l-1}
        8 l^{-\alpha} (k^{-\alpha} + c_1 (l-k)^{-1 -\tau\varepsilon})
        \nonumber \\
      & \leq S_j + 2 \sum_{l = n_j + 1}^{n_{j+1}} 8 l^{-\alpha} \biggl(
        \int_0^l x^{-\alpha} \, \mathrm{d}x + c_2 \biggr) \nonumber \\
      & = S_j + 16 \sum_{l = n_j + 1}^{n_{j+1}} l^{-\alpha} \biggl(
        \frac{1}{1 - \alpha} l^{1 - \alpha} + c_2 \biggr) \nonumber \\
      & \leq S_j + \frac{16}{(1 - \alpha)^2}
        (n_{j+1} + 1)^{2 - 2\alpha} + \frac{16 c_2}{(1 - \alpha)}
        (n_{j+1} + 1)^{1 - \alpha}. \label{eq:est1}
  \end{align}
  We also have
  \begin{equation} \label{eq:est2}
    S_j = \sum_{k = n_j}^{n_{j+1}} 2 k^{-\alpha} \geq 2
    \int_{n_j}^{n_{j+1}} x^{-\alpha} \, \mathrm{d} x = \frac{2}{1
      - \alpha} (n_{j+1}^{1 - \alpha} - n_j^{1 - \alpha}).
  \end{equation}

  For any such $j$, we then have by the very famous and important
  Chung--Erd\H{o}s inequality \cite{MR45327} that
  \[
    \lambda \biggl( \bigcup_{k = n_j}^{n_{j+1}} A_k (y) \biggr)
    \geq \frac{S_{n_j,n_{j+1}}^2}{C_{n_j,n_{j+1}}} \geq
    \frac{S_j^2}{p C_j}.
  \]
  By the estimates \eqref{eq:est1} and \eqref{eq:est2} we have
  that
    \[
    \lim_{j \to \infty} \frac{S_j^2}{C_j} = 
    \frac{1}{4p}.
  \]
  It follows that for any $y \in G (p)$ we have
  $\lambda(W_{y,\alpha}) \geq \frac{1}{4p} > 0$. Hence
  $\lambda (W_{y,\alpha}) > 0$ for any $y \in \bigcup_{p>1} G(p)$
  and $\mu (\bigcup_{p>1} G(p)) = 1$.
\end{proof}

Now, we consider the special case of $\mu = \lambda$ and prove
Theorem~\ref{thm:Lmeasure}. For this, we will use an inflation argument due to Cassels, which in turn needs Lemma~VII.5 from \cite{MR349591}, which we state here. In  \cite{MR349591} the lemma is proved in one dimension, but as remarked in that book, the proof in higher dimensions is identical. We state it here as Lemma \ref{lem:Cassels} in two dimensions, as this is the case we will need.

\begin{lem}
\label{lem:Cassels}
Let $A \subset [0,1)^2$ be a set of positive Lebesgue measure
and let $\epsilon > 0$. There are integers $t_1, t_2$ and $T$
with $0 \le t_1, t_2 < T$ such that the Lebesgue measure of the
set
\[ 
\{(x,y) \in A : t_1/T \le x < (t_1 +1)/T, \;  t_2/T \le y <(t_2 +1)/T\}
\]
is at least $(1-\epsilon)/T^2$.
\end{lem}

\begin{proof}[Proof of Theorem~\ref{thm:Lmeasure}] 
We first note that $|\hat{\lambda} (\xi)| = O (|\xi|^{-1})$, whence by Theorem \ref{thm:measure}, $W_{y,\alpha}$ will have positive Lebesgue measure for any $\alpha < 1$.

We consider the auxiliary set 
\[
  \overline{W}_\alpha = \{ \, (x,y) \in [0,1]^2 : \Vert q_n
  y - x\Vert < n^{-\alpha} \text{ infinitely often } \, \}.
\]
By Theorem~\ref{thm:measure} and Fubini's Theorem, the two
dimensional Lebesgue measure of this set is positive for any
$\alpha < 1$.

Now, fix $\alpha'  \in (\alpha, 1)$. By the above, $\lambda(\overline{W}_{\alpha'}) > 0$. We will prove that this implies that $\lambda(\overline{W}_\alpha) = 1$. 

Let $\epsilon > 0$. We split up the square $[0,1)^2$ into $T^2$ disjoint squares of side length $1/T$, where $T \in \mathbb{N}$ is chosen by Lemma~\ref{lem:Cassels}, so that for one of these boxes, $B$ say, 
\[
\lambda(\overline{W}_{\alpha'} \cap B) > (1-\epsilon) \lambda(B).
\]
Scaling the box by $T^2$ and projecting back onto the $[0,1)^2$ by reducing modulo $1$, we cover the entire unit square, so that
\[
\lambda(T(\overline{W}_{\alpha'} \cap B) \pmod{1}) > (1-\epsilon).
\]
Since $\epsilon > 0$ was arbitrary, it follows that for every pair $(x,y) \in [0,1)^2$ outside of a set of Lebesgue measure $< \epsilon$, there is a point $(x',y') \in \overline{W}_{\alpha'}$ with $(x,y) = T(x',y') \pmod{1}$.

Now, for such a point, 
\[
\Vert q_n y - x \Vert = \Vert T(q_n y' - x') \Vert \le T \Vert q_n y' - x' \Vert.
\]
Since $(x',y') \in \overline{W}_{\alpha'}$, for infinitely many values of $n$, 
\[
T \Vert q_n y' - x' \Vert < T n^{-\alpha'} = T n^{\alpha - \alpha'} n^{-\alpha}.
\]
Since $\alpha' > \alpha$, $T n^{\alpha - \alpha'} \le 1$ for $n$ large enough, since $T$ was fixed. Consequently, $(x,y) \in \overline{W}_{\alpha'}$, so that $\lambda(\overline{W}_{\alpha}) > 1-\epsilon$. Since $\epsilon > 0$ was arbitrary, it follows that $\lambda(\overline{W}_{\alpha}) = 1$.

Now, to conclude we apply Fubini's Theorem once more and find that almost all fibers $W_{y,\alpha}$ are Lebesgue full.
\end{proof}

We now prove Theorem~\ref{thm:Hdim}. In the proof, we will make
use of the following quantity related to discrepancy,
\[
  \tilde{D}_N(x_n) = \sup_{I \subseteq [0,1)} \vert \# \{\, N+1 \leq
  n \leq 2N : x_n \in I \,\} - N\,\lambda(I)\vert.
\]
We will need the following lemmata.

\begin{lem}
  If $D_N ((x_n)) \leq O (N^{1-\eta})$ then
  $\tilde{D}_N ((x_n)) \leq O (N^{1-\eta})$.
\end{lem}

\begin{proof}
  Take $I \subset [0,1)$. Then
  \begin{align*}
    \biggl| \sum_{n=N+1}^{2N} \mathbbm{1}_I (x_n) - N|I| \biggr|
    &= \biggl| \sum_{n=1}^{2N} \mathbbm{1}_I (x_n) - 2N|I| -
      \sum_{n=1}^{N} \mathbbm{1}_I (x_n) + N|I| \biggr| \\
    & \leq D_{2N} ((x_n)) + D_N ((x_n)). \qedhere
  \end{align*}
\end{proof}

We will also need the following result of Persson and Reeve
\cite{MR3341452}.

\begin{lem}[Persson and Reeve \cite{MR3341452}]
  \label{lem:perssonreeve}
  Let $(\mu_n)$ be a sequence of probability measures with
  $\supp \mu_n \subset E_n \subset [0,1]$, and such that
  $\mu_n \to \lambda$ weakly. If there is a constant $C$ such
  that
  \[
    I_s (\mu_n) := \iint |x-y|^{-s} \, \mathrm{d}\mu (x)
    \mathrm{d}\mu (y) < C
  \]
  for all $n$, then the set $E = \limsup E_n$ has Hausdorff
  dimension at least $s$, and $E$ has large intersections as
  defined in Section~\ref{sec:background}.
\end{lem}

The reader will have noted that the upper bound on the dimension
in each case is the same. It is the easier part of all the
statements. We prove it here as a lemma.

\begin{lem}
  \label{lem:upper}
  Let $(x_n)_{n=1}^\infty$ be any sequence in $[0,1)$ and let $\alpha \ge 1$. Then,
  \[
    \hdim \{\, \gamma \in [0,1) : \Vert x_n - \gamma \Vert <
    n^{-\alpha} \text{ infinitely often} \,\} \le \frac{1}{\alpha}.
  \]
\end{lem}

\begin{proof}
  Let $s = 1/\alpha + \epsilon$ for some $\epsilon > 0$. For any
  $N \in \mathbb{N}$, the set, $W$ say, in the statement of the
  lemma is covered by the intervals
  $(x_n - n^{-\alpha}, x_n + n^{-\alpha})$, each of length
  $2n^{-\alpha}$. For $\delta > 0$, pick $N_0 \in \mathbb{N}$ so
  that $2n^{-\alpha} < \delta$ for $n \ge N_0$. Then, for any
  $N \ge N_0$
  \[
    \mathcal{H}^s_\delta (W) \le \sum_{n=N}^\infty
    (2n^{-\alpha})^{s} = 2^{-1-\alpha\epsilon} \sum_{n=N}^\infty
    n^{-1-\alpha\epsilon}.
  \]
  The latter tends to $0$ as $N$ tends to infinity, so for any
  $\delta > 0$, $\mathcal{H}^s_\delta (W) = 0$, whence
  $\mathcal{H}^s (W) = 0$. But then,
  $\hdim (W) \le s = 1/\alpha + \epsilon$. As $\epsilon> 0$ was
  arbitrary, we are done.
\end{proof}

\begin{proof}[Proof of Theorem~\ref{thm:Hdim}]
  By Lemma~\ref{lem:upper}, we need only prove the lower bound.
  We let $\mu_N$ be the sequence of probability measures defined
  by
  \[
    \frac{\mathrm{d} \mu_N}{\mathrm{d}x} = \frac{1}{N}
    \sum_{k=N+1}^{2N} \frac{1}{2(2N)^{-\alpha}}
    \mathbbm{1}_{B(a_k x, r_N)}, \qquad r_N = (2N)^{-\alpha},
  \]
  and consider
  \[
    I_t (\mu_N) = \iint |x-y|^{-t} \, \mathrm{d}\mu (x)
    \mathrm{d}\mu (y) = c_0 \int |\hat{\mu}_N (\xi)|^2 |\xi|^{t-1} \,
    \mathrm{d} \xi,
  \]
  where $c_0$ is a constant which does not depend on $N$ and $t$.

  We have
  \[
    \frac{\mathrm{d} \mu_N}{\mathrm{d}x} = \frac{1}{N}
    \sum_{k=N+1}^{2N} \frac{1}{2(2N)^{-\alpha}} \mathbbm{1}_{B(0,
      r_N)} * \delta_{a_k x},
  \]
  so that
  \[
    \hat{\mu}_N (\xi) = (2N)^{\alpha-1} \sum_{k=N+1}^{2N}
    \widehat{\mathbbm{1}_{B(0, r_N)}} (\xi) e^{-i 2 \pi \xi a_k x},
  \]
  and 
  \[
    |\hat{\mu}_N (\xi)|^2 = (2N)^{2\alpha - 2} \sum_{k=N+1}^{2N}
    \sum_{l=N+1}^{2N} |\widehat{\mathbbm{1}_{B(0, r_N)}} (\xi)|^2 e^{-i
      2 \pi \xi (a_k - a_l) x}.
  \]

  Since
  \[
    \widehat{\mathbbm{1}_{B(0, r_N)}} (\xi) = r_N \frac{\sin(2\pi r_n
      \xi)}{\pi r_N \xi},
  \]
  we can estimate that
  \[
    |\widehat{\mathbbm{1}_{B(0, r_N)}} (\xi)|^2 \leq m_N (\xi) := \min
    \{2 r_N, \pi |\xi|^{-1} \}^2.
  \]
  It follows that
  \begin{align*}
    \int |\hat{\mu}_N (\xi)|^2 \, \mathrm{d}\mu
    & \leq (2N)^{2\alpha-2} \sum_{k=N+1}^{2N} \sum_{l=N+1}^{2N} \int m_N
      (\xi) e^{-i 2 \pi \xi (a_k - a_l) x} \, \mathrm{d} \mu (x) \\
    & = (2N)^{2\alpha-2} m_N (\xi) \sum_{k=N+1}^{2N} \sum_{l=N+1}^{2N}
      \hat{\mu} ((a_k - a_l) \xi).
  \end{align*}

  Take constants $c$ and $\eta > 0$ such that
  $|\hat{\mu} (\xi)| \leq c \min \{1, |\xi|^{-\eta} \}$. Then, we
  have
  \[
    \sum_{k=N+1}^{2N} \sum_{l=N+1}^{2N} \hat{\mu} ((a_k - a_l)
    \xi) \leq c N^2.
  \]
  If $|\xi| \geq 1$, we can use that $(a_k)$ is lacunary to prove
  that
  \[
    \sum_{k=N+1}^{2N} \sum_{l=N+1}^{2N} \hat{\mu} ((a_k - a_l)
    \xi) \leq \sum_{k=N+1}^{2N} \sum_{l=N+1}^{2N} c \min \{1,
    |a_k - a_l|^{-\eta}\} \leq C N,
  \]
  for some constant $C$. We thus have
  \[
    \int |\hat{\mu}_N|^2 \, \mathrm{d}\mu \leq c (2N)^{2 \alpha}
    m_N (\xi),
  \]
  and if $|\xi| \geq 1$, we have
  \[
    \int |\hat{\mu}_N|^2 \, \mathrm{d}\mu \leq C (2N)^{2\alpha -
      1} m_N (\xi).
  \]

  Combining the last two estimates, we get that
  \begin{align*}
    \int I_t (\mu_N) \, \mathrm{d} \mu
    & = c_0 \iint |\hat{\mu}_N (\xi)|^2 |\xi|^{t-1} \, \mathrm{d}
      \xi \, \mathrm{d} \mu \\
    &= c_0 \iint |\hat{\mu}_N (\xi)|^2 \, \mathrm{d} \mu\, |\xi|^{t-1} \,
      \mathrm{d} \xi \\
    &\leq c' N^{2\alpha} \int_0^1 m_N (\xi)
      |\xi|^{t-1} \, \mathrm{d} \xi + C' N^{2 \alpha - 1} \int_1^\infty
      m_N (\xi) |\xi|^{t-1} \, \mathrm{d} \xi.
  \end{align*}
  We can now easily compute that
  \begin{align*}
    N^{2\alpha} \int_0^1 m_N (\xi) |\xi|^{t-1} \, \mathrm{d} \xi
    &\leq N^{2\alpha} r_N^2 \int_0^1 \xi^{t-1} \, \mathrm{d} \xi =
      c_1,\\ N^{2\alpha-1} \int_1^{\frac{\pi}{2} r_N^{-1}} m_N (\xi)
    |\xi|^{t-1} \, \mathrm{d} \xi
    & = 4 N^{2\alpha-1} r_N^2 \int_1^{\frac{\pi}{2} r_N^{-1}}
      \xi^{t-1} \, \mathrm{d}\xi \leq c_2 N^{- 1 + \alpha t}, \\
    N^{2\alpha-1} \int_{\frac{\pi}{2}
    r_N^{-1}}^\infty m_N (\xi) |\xi|^{t-1} \, \mathrm{d} \xi
    &= 2
      \pi^2 C N^{2\alpha-1} \int_{\frac{\pi}{2} r_N^{-1}}^\infty
      \xi^{t-3} \, \mathrm{d}\xi \leq c_3 N^{-1 + \alpha t},
  \end{align*}
  where $c_1$, $c_2$ and $c_3$ do not depend on $N$. Hence
  \[
    \int I_t (\mu_N) \, \mathrm{d}\mu \leq c_4 + c_5 N^{-1 +
      \alpha t},
  \]
  which is uniformly bounded in $N$ provided $t < 1/\alpha$. 

  Fix $t < 1/\alpha$. It follows that for $\mu$-almost all $x$,
  there is a subsequence $(N_k)$ of the natural numbers, such
  that $I_t (\mu_{N_k})$ is uniformly bounded in
  $k$. Lemma~\ref{lem:perssonreeve} finishes the proof.
\end{proof}

\begin{proof}[Proof of Theorem~\ref{thm:Hdim2}]
  As before, the upper bound is taken care of by Lemma~\ref{lem:upper}. Once again, we define probability measures
  $\mu_N$ with densities
  \[
    \frac{\mathrm{d} \mu_N}{\mathrm{d}x} = \frac{1}{N}
    \sum_{n=N+1}^{2N} \frac{1}{2(2N)^{-\alpha}} \mathbbm{1}_{B(x_n,
      (2N)^{-\alpha})}.
  \]
  Then
  \[
    \supp \mu_N = E_N = \bigcup_{n=N+1}^{2N} B(x_n,
    (2N)^{-\alpha}).
  \]
  Hence $E_N \subset W_N$, where
  \[
    W_N = \bigcup_{n=N+1}^{2N} B(x_n, n^{-\alpha}).
  \]
  It follows that $\limsup E_N \subset W = \limsup W_N$.

  Let $s = \eta/\alpha$. We show that there is a constant $c_1$ such that
  \[
    \mu_N (I) \leq c_1 |I|^s
  \]
  holds for every $N$ and every interval $I$ (and hence for every
  Borel set).
  
  Let $M (I,N)$ denote the number of $n$ with $N + 1 \leq n \leq 2N$
  and $x_n \in I$. We have
  \[
    M (N,I) \leq N |I| + \tilde{D}_N ((x_n)) \leq N |I| + c
    N^{1-\eta}.
  \]

  In every point $y \in [0,1)$, the number of $n = N+1, \ldots,
  2N$ such that $y \in B(x_n, (2N)^{-\alpha})$ is at most
  \[
    \sup_{|I|=2(2N)^{-\alpha}} M (I,N) \leq 2^{1-\alpha}
    N^{1-\alpha} + c N^{1-\eta} < (1 + c) N^{1-\eta}.
  \]
  Hence, the density of $\mu_N$ satisfies
  \begin{equation}
    \label{eq:density}
    \frac{\mathrm{d} \mu_N}{\mathrm{d} x} \leq \frac{1}{N}
    \frac{1}{2(2N)^{-\alpha}} (1+c) N^{1-\eta} \leq 2^{\alpha - 1} (1 +
    c) N^{\alpha - \eta}.
  \end{equation}
  
  If $|I| < N^{-\alpha}$, we use \eqref{eq:density} to estimate that
  \begin{align*}
    \mu_N (I) &\leq 2^{\alpha - 1} (1+c) N^{\alpha - \eta} |I| \\
              &= 2^{\alpha - 1} (1+c) N^{\alpha - \eta}
                |I|^{1-\eta/\alpha} |I|^{\eta/\alpha} \leq
                2^{\alpha - 1} (1+c) |I|^{\eta/\alpha}.
  \end{align*}
  
  If $|I| \geq N^{-\alpha}$, let $3I$ denote the interval
  concentric with $I$ and such that $|3I| = 3 |I|$. Then
  \begin{equation}
    \label{eq:muNestimate}
    \mu_N (I) \leq \frac{1}{N} M(3I, N) \leq 3|I| + c N^{-\eta}
    \leq 3 |I| + c |I|^{\eta/\alpha} \leq (3 + c)
    |I|^{\eta/\alpha}.
  \end{equation}

  This shows that $\mu_N (I) \leq c_2 |I|^s$ with
  $c_2 = 2^{\alpha - 1} (3+c)$.

  It now follows for $t < s$ that
  \begin{align*}
    I_t (\mu_N) &= \iint |x-y|^{-t} \, \mathrm{d} \mu_N (x) \mathrm{d}
                  \mu_N (y) \\
                &= \int \int_0^\infty \mu_N (B(x,u^{-1/t})) \, \mathrm{d}u
                  \, \mathrm{d} \mu_N(x) \\
                &\leq \int \int_0^\infty \min \{1, c_2
                  (u^{-1/t})^s \} \, \mathrm{d}u \, \mathrm{d}
                  \mu_N (x) \\
                &= \int_0^\infty \min \{ 1, c_2 u^{-s/t} \} \,
                  \mathrm{d}u =: C < \infty.
  \end{align*}

  Since $D_N ((x_n)) \leq o (N)$, the measures $\mu_N$ converges
  weakly to $\lambda$ and it follows from
  Lemma~\ref{lem:perssonreeve} that $\dimh \limsup E_N \geq
  t$. Since $t$ can be taken arbitrary close to
  $s = \eta /\alpha$ we have
  $\dimh W \geq \dimh \limsup E_N \geq \eta/\alpha$.
\end{proof}

\begin{proof}[Proof of Theorem~\ref{thm:extraassumption}]
  The proof is similar to that of Theorem~\ref{thm:Hdim2}. We
  define $\mu_N$ in the same way. It is then enough to prove that
  there is a constant $c_3$ such that
  $\mu_N (I) \leq c_3 |I|^\frac{1}{\alpha}$ holds for all $N$ and
  all intervals $I$. We consider three cases.

  First, we assume that $|I| \leq N^{-\beta}$. Then $I$
  intersects at most two balls $B(x_n,
  (2N)^{-\alpha})$. Therefore, we have
  \[
    \mu_N (I) \leq |I| \frac{2}{N} \frac{1}{2(2N)^{-\alpha}} \leq
    2^\alpha |I| N^{\alpha - 1} \leq 2^\alpha |I|^{1 -
      \frac{1}{\alpha} (\alpha - 1)} = 2^\alpha
    |I|^\frac{1}{\alpha}.
  \]

  If $\beta > \eta \alpha$, then we assume that
  $N^{-\beta} \leq |I| \leq N^{-\eta \alpha}$. Otherwise, we do
  not have to consider this case.

  The assumption $\beta \leq \eta (\alpha - 1) + 1$ implies that
  $\alpha \geq \beta$. Hence, $I$ intersects a ball
  $B(x_n, (2N)^{-\alpha})$ only if $x_n \in 3 I$. The number of
  such $x_n$ is at most $3 |I| / (2N)^{-\beta}$ and
  \[
    \mu_N (I) \leq 2^\beta 3 |I| N^\beta \frac{1}{N} = 2^\beta 3
    |I|^\frac{1}{\alpha} |I|^{1-\frac{1}{\alpha}} N^{\beta - 1}
    \leq 2^\beta 3 |I|^\frac{1}{\alpha} N^{- \eta (\alpha - 1) +
      \beta - 1}.
  \]
  Since $- \eta (\alpha - 1) + \beta - 1 \leq 0$ by assumption,
  we have $\mu_N (I) \leq 2^\beta 3 |I|^\frac{1}{\alpha}$.

  Finally, assume that $|I| \geq N^{-\eta \alpha}$. Then we have
  as in \eqref{eq:muNestimate} that
  \[
    \mu_N (I) \leq 3 |I| + c N^{-\eta} \leq (3 + c)
    |I|^\frac{1}{\alpha}.
  \]

  All estimates taken together, we have for any interval $I$ and
  any $N$ that $\mu_N (I) \leq c_3 |I|^\frac{1}{\alpha}$, with
  $c_3 = \max \{2^\alpha, 2^\beta 3, 3+c \}$. The rest of the
  proof is exactly as in the proof of Theorem~\ref{thm:Hdim2}.
\end{proof}

\begin{proof}[Proof of Corollary~\ref{cor:lastcor}]
  Let $\eta < \frac{1}{2}$. Then $D_N ((q_ny)) \leq c N^{1-\eta}$
  for $\lambda$-a.e. $y \in [0,1)$.
  
  Let $\beta > 2$ and consider the set
  \[
    \Delta = \{\, |q_l - q_k| : k \neq l \,\}.
  \]
  Let $(n_j)_{j=1}^\infty$ be an enumeration of $\Delta$ such
  that if $q_l - q_k \in \Delta$, then there is a $j \leq l^2$
  such that $n_j = q_l - q_k$. (Such an enumeration clearly
  exists.)

  Consider the set
  \[
    V = \{\, y \in [0,1) : \lVert n_j y \rVert <
    j^{-\beta/2} \text{ for infinitely many } n \in \mathbb{N}
    \,\}.
  \]
  Since for each $j$,
  \[
    \lambda \{\, y \in [0,1) : \lVert n_j y - \gamma \rVert <
    j^{-\beta/2} \,\} = 2 j^{-\beta/2},
  \]
  we have by the first Borel--Cantelli lemma that
  $\lambda (V) = 0$. It follows that for $\lambda$-a.e.\ $y$
  there exists $c > 0$ such that
  $\lVert n_j y \rVert > c j^{-\beta/2}$ for all $j$.

  Suppose now that $y$ is such a typical point, and consider the
  sequence $(q_n y)$. Take $1 \leq k < l \leq N$ and let
  $j \leq l^2$ be such that $n_j = q_l - q_k \in \Delta$. Since
  $j \leq N^2$, and by the choice of $y$, we have
  \[
    \lVert q_k y - q_l y \rVert = \lVert (q_l - q_k) y \rVert =
    \lVert n_j y \rVert > c j^{-\beta/2} \geq c N^{-\beta}.
  \]

  This shows that $d_N((q_ny)) > cN^{-\beta}$. It follows by
  Theorem~\ref{thm:extraassumption} that
  $\dimh W_{y,\alpha} = \frac{1}{\alpha}$ whenever $\alpha$ is
  such that $\beta \leq \eta (\alpha - 1) + 1$. Since $\beta$ can
  be taken as close to $2$ and $\eta$ as close to $\frac{1}{2}$
  as we desire, $\dimh W_{y,\alpha} = \frac{1}{\alpha}$ holds for
  almost all $y$ as long as $\alpha > 3$. The case $\alpha = 3$
  follows by a limiting proceedure.
\end{proof}

\section{Concluding remarks}

We strongly suspect that --- as in the case of Lebesgue measure
--- the conclusion of Theorem~\ref{thm:measure} can be
strengthened from positive to full measure. However, the
`inflation argument' of Cassels used in the proof of Theorem \ref{thm:Lmeasure} is not immediately applicable
unless both measures are Lebesgue or at least have the same
scaling properties. Hence, we do not at present know how to
extend this.

In Corollary \ref{cor:lastcor}, the exact value of the Hausdorff dimension is calculated, but only for $\alpha \ge 3$. Again, it would be natural to suspect that this can be extended to $\alpha > 1$, but our methods do not allow us to do this.

As a final remark, in the paper of Pollington, Velani, Zafeiropoulos and Zorin \cite{MR4425845}, a weaker requirement on the Fourier decay of the measure $\mu$ is required. It is possible that positive Fourier dimension can be weakened to, say, logarithmic decay. We have not attempted to do this in the present manuscript.

\providecommand{\bysame}{\leavevmode\hbox to3em{\hrulefill}\thinspace}
\providecommand{\MR}{\relax\ifhmode\unskip\space\fi MR }
\providecommand{\MRhref}[2]{%
  \href{http://www.ams.org/mathscinet-getitem?mr=#1}{#2}
}
\providecommand{\href}[2]{#2}

\end{document}